\documentclass[preprint,12pt]{elsarticle}
\usepackage{tikz}
\usepackage{amsthm,amsfonts,amssymb,amscd,amsmath,enumerate,verbatim,calc,graphicx,geometry}
\usepackage[all]{xy}
\newtheorem{theorem}{Theorem}[section]

\newtheorem{proposition}[theorem]{Proposition}
\newtheorem{corollary}[theorem]{Corollary}
\theoremstyle{definition}
\theoremstyle{definitions}

\newtheorem{remark}[theorem]{Remark}
\newtheorem{example}[theorem]{Example}
\theoremstyle{notations}

\theoremstyle{remarks}

\newcommand{\ov}{\overline}

\journal{}

\begin{document}

\begin{frontmatter}



\title{Comparison of Topologies on Fundamental Groups with Subgroup Topology Viewpoint }


\author[]{Naghme Shahami }
\ead{naghme.shahami75@gmail.com }
\author[]{Behrooz Mashayekhy\corref{cor1}}
\ead{bmashf@um.ac.ir }

\address{Department of Pure Mathematics, Center of Excellence in Analysis on Algebraic Structures, Ferdowsi University of Mashhad,\\
P.O.Box 1159-91775, Mashhad, Iran.}
\cortext[cor1]{Corresponding author}

\begin{abstract}
In order to make the fundamental group, one of the most well known invariants in algebraic topology, more useful and powerful some researchers have introduced and studied various topologies on the fundamental group from the beginning of the 21st century onwards. In this paper by reviewing these topologies, using the concept of subgroup topology, we are going to compare these topologies in order to present some results on topologized fundamental groups.
   
\end{abstract}


\end{frontmatter}



\section{Introduction and Motivation}
Historically, putting a natural topology on the fundamental group comes back to Hurewicz \cite{Hur} in 1935 and Dugundji \cite{D} in 1950. From the beginning of the 21st century some researchers have been shown that there are various useful, interesting and functorial topologies on the fundamental group which make it a powerful tool for studying some  topological properties of spaces (see \cite{A16,A20,B11,B12,B13,B14,B15,CM,F9,F11,J,MPT,PTM17,PTM20,TPM,VZ14}).  Some well-known topologies on the fundamental group $\pi_1(X,x_0)$ are as follows.

1. \textbf{The whisker topology}: Let $(X,x_0)$ be a pointed topological space. Define $\widetilde{X} = \{[\alpha ] |  \alpha :I \rightarrow X , \alpha (0) = x_0\}$ as the set of all path-homotopy classes of paths in $X$ starting at $x_0$. Spanier \cite[p. 82]{Span} introduced a topology on $\widetilde{X}$ by basic open neighbourhoods of $[\alpha]$ of the form $N([\alpha],U) = \{[\alpha \ast \delta] | \delta : I \rightarrow U , \delta (0)=\alpha (1) \}$, where $U$ is an open neighbourhood of $\alpha (1)$ in $X$. This topology on $\widetilde{X}$ is called Brodskiy et al. \cite{Br12} \textit{the whisker topology} due to $N([\alpha],U)$ consists of only homotopy classes that differ from$ [\alpha]$ by a small change or “whisker” at it’s end.
The function $p:\widetilde{X} \to X$ as the endpoint projection $p([\alpha])=\alpha (1)$ is continuous and if $X$ is connected, locally path connected, and semilocally simply connected then $p$ is the universal covering map. Clear $p^{-1}(\{x_0\})=\pi_1(X,x_0)$ and so $\pi_1(X,x_0)$ is a subspace of $\widetilde{X}$. One can consider the subspace topology on $\pi_1(X,x_0)$  inherited from $\widetilde{X}$ which is called \textit{the whisker topology} on the fundamental group denoted by $\pi_1^{\mathrm{wh}}(X, x_0)$  (see \cite{Br12}).

2. \textbf{The compact-open quotient topology}:  Let $(X,x_0)$ be a pointed topological space and $\Omega (X,x_0)$ denote the space of loops in X based at $ x_0$. There exists the usual compact-open topology on $\Omega (X,x_0)$ which is generated by subbasis sets $\langle K,U\rangle=\{\alpha\mid \alpha(K)\subseteq U\}$ for compact $K\subseteq [0,1]$ and open $U\subseteq X$. By considering the surjection map $q:\Omega(X,x_0)\to\pi_1(X,x_0), q(\alpha)=[\alpha]$ one can equips $\pi_1(X,x_0)$  with the quotient topology with respect to the map $q:\Omega(X,x_0)\to\pi_1(X,x_0)$ which id denoted by $\pi_{1}^{qtop}(X,x_0)$. We refer to this topology as the \textit{natural quotient topology} on $\pi_{1}(X,x_0)$ (see \cite{Biss, B11, F9}).

3. \textbf{The lasso topology}: For any topological space $X$, Brodskiy et al. \cite[Section 3]{Br8} introduced the lasso topology on the universal path space $\widetilde{X}$ by the basis $  N(\langle \alpha \rangle , \mathcal{U}, W)$, where $ \alpha $ is a path originated at $ x_0 $,  $W$ is a neighbourhood of the endpoint $\alpha(1)$ and $\mathcal{U}$ is an open cover of $X$. A class $\langle \gamma \rangle \in \widetilde{X}$ belongs  to $ N(\langle \alpha \rangle , \mathcal{U}, W)$ if and only if this class has a representative of the form $\alpha * L * \beta$ where $L$ belongs to $\pi\big(\mathcal{U},\alpha(1)\big)$ and $\beta$ is a based loop in $W$ at $\alpha(1)$. There is a bijection between the fundamental group $\pi_1(X, x_0)$ and the fibre of the base point $p^{-1}(x_0)$, where $p: \widetilde{X}\to X$ is the endpoint projection map. Therefore, the fundamental group $\pi_1(X, x_0)$ as a subspace of the universal path space $\widetilde{X}$ inherits any topology from $\widetilde{X}$. Thus, the collection of sets with the form $ N(\langle \alpha\rangle, \mathcal{U}, W)\cap p^{-1}(x_0)$ is a basis for the lasso topology on $\pi_1(X, x_0)$, which is denoted by $\pi_1^{\mathrm{lasso}} (X, x_0)$ (see \cite[Definition 4.11]{Br12}).

4. \textbf{The $\tau$-topology}:  Brazas in \cite{B13} proved that there exists the finest topology on $\pi_{1}(X,x_{0})$ such that $\pi: \Omega(X,x_{0})\rightarrow \pi_{1}(X,x_{0})$ is continuous and $\pi_{1}(X,x_{0})$ is a topological group. The fundamental group $\pi_{1}(X,x_{0})$ with this topology is denoted by $\pi_1^{\mathrm{\tau}}(X, x_0)$ (see \cite{B13}).

5. \textbf{The Spanier topology}:  The collection $\Sigma$ of subgroups of a group $G$ is called a \textit{neighbourhood family} if for any $H,K \in \Sigma$, there is a subgroup $S \in \Sigma$ such that $S \subseteq H\cap K$. The collection of all left cosets of elements of $\Sigma$ forms a basis for a topology on $G$, which is called the subgroup topology determined by $\Sigma$. As an example, \textit{Spanier subgroup topology} \cite[p. 12]{W} was introduced using the collection of all Spanier subgroups $\pi(\mathcal{U}, x_0)$ of the fundamental group $\pi_1(X, x_0)$ as the neighbourhood family $ \Sigma^{Span} $. Recall that \cite[p. 81]{Span}, the Spanier subgroup determined by an open covering $\mathcal{U}$ of $X$ is the normal subgroup $\pi(\mathcal{U}, x_0)$ of $\pi_1(X, x_0)$ generated by the homotopy class of lollipops $\alpha * \beta * \alpha^{-1}$, where $\beta$ is a loop lying in an element of $U \in \mathcal{U}$ at $\alpha(1)$, and $\alpha$ is any path originated at $x_0$. The fundamental group equipped with the Spanier subgroup topology is denoted by $\pi_1 ^{\mathrm{Span}}(X, x_0)$ (see \cite{A20,W}).

6. \textbf{The path Spanier topology}: Torabi et al. \cite[Section 3]{T} replaced open covers with path open covers of the space $ X $ in the definition of Spanier subgroups and introduced path Spanier subgroups by the same way. Recall that a path open cover $ \mathcal{V} $ of the path component of $ X $ involve $ x_0 $ is the collection of open subsets $\lbrace V_{\alpha} \ \vert \ \alpha\in{P(X,x_{0})}\rbrace$ and the path Spanier subgroup $\widetilde{\pi}(\mathcal{V}, x_{0})$ with respect to the path open cover $ \mathcal{V} $ is the subgroup of $ \pi_1(X, x_0)$ consists of all homotopy classes having representatives of the following type:
$$\prod_{j=1}^{n}\alpha_{j}\beta_{j}\alpha^{-1}_{j},$$
where $\alpha_{j}$'s are arbitrary path starting at $x_{0}$ and each $\beta_{j}$ is a loop inside of the open set $V_{\alpha_{j}}$ for all $j\in{\lbrace1,2,...,n\rbrace}$. Note that the path Spanier subgroup $\widetilde{\pi}(\mathcal{V}, x_{0})$ is not a normal subgroup, in general (see \cite[Example 3.7]{A20}).
If $ \mathcal{U} $ and $ \mathcal{V} $ are two path open covers of a space $ X $, the collection $ \mathcal{W}=\lbrace U_{\alpha} \cap V_{\alpha} \ \vert \ \forall \alpha\in{P(X,x_{0}), U_{\alpha} \in \mathcal{U} \ and \ V_{\alpha} \in \mathcal{V} }\rbrace $ is a refinement of both $ \mathcal{U} $ and $ \mathcal{V} $. Thus, $\widetilde{\pi}(\mathcal{W}, x_{0}) \leq  \widetilde{\pi}(\mathcal{U}, x_{0}) \cap \widetilde{\pi}(\mathcal{V}, x_{0}) $, which shows that the collection of all path Spanier subgroups of the fundamental group forms a neighbourhood family.
For a pointed topological space $ (X,x_0) $, let $ \Sigma^{\mathrm{pSpan}} $ be the collection of all path Spanier subgroups of  $ \pi_1(X, x_0) $. The subgroup topology determined by $\Sigma^{\mathrm{pSpan}} $ is called the \textit{path Spanier topology} which is denoted by $ \pi_1^{\mathrm{pSpan}}(X, x_0)$ (see \cite{A20,T}).

7. \textbf{The generalized covering topology}: Brazas \cite{B15} introduced a generalized covering subgroup for a pointed topological space $ (X,x_0) $. Let $\Sigma^{\mathrm gcov} $ be the collection of all generalized covering subgroups of  $\pi_1(X, x_0)$. The subgroup topology determined by $\Sigma^{\mathrm gcov} $ is called the \textit{generalized covering topology} and denoted by $\pi_1^{\mathrm{gcov}}(X, x_0)$ (see \cite{A16,A20,B15,FZ7}).

8. \textbf{The shape topology}: The first shape homotopy group of a pointed topological space $(X,x_0)$ is the inverse limit
$$ \check{\pi}_1(X,x_0)=\varprojlim (\pi_1(|N(\mathcal{U})|,U_0),p_{\mathcal{U}\mathcal{V}*},\Lambda),  $$
where the inverse system $\pi_1(|N(\mathcal{U})|,U_0),p_{\mathcal{U}\mathcal{V}*},\Lambda)$ of discrete groups is the first pro-homotopy group topologized with the usual inverse limit topology. The \emph{shape topology} on $\pi_1(X, x_0)$ is the initial topology with respect to the first shape homomorphism $\Psi_1: \pi_1(X, x_0)\rightarrow \check{\pi}_1(X,x_0)$.
Let $\pi_1^{sh}(X, x_0)$ denote $\pi_1(X, x_0)$ equipped with the shape topology group. Note that one can consider the shape topology on $\pi_1(X, x_0)$ as a subgroup topology with respect to all kernels of homomorphisms induced by maps $X\rightarrow K$ to simplicial complexes $K$ (see \cite{B13, BF14,  BF24} for more details).

9. \textbf{The thick Spanier topology}:  In order to study the kernel of the first shape homomorphism $\Psi_1$, Brazs and Fabel \cite{BF14} introduced the thick Spanier subgroup of $X$ with respect to an open cover $\mathcal{U}$ of $X$, denoted by $\Pi^{sp}(\mathcal{U}, x_0)$, with a modification in the definition of Spanier subgroups, as a subgroup of $\pi_1(X, x_0)$ generated by the homotopy classes of $\alpha * \beta *\gamma* \alpha^{-1}$, where $\beta$ and $\gamma$ are paths lying in elements of $U_1,U_2 \in \mathcal{U}$, respectively, and $\alpha$ is any path originated at $x_0$. The collection of all thick Spanier subgroups $\Pi^{sp}(\mathcal{U}, x_0)$ of the fundamental group $\pi_1(X, x_0)$ forms a neighbourhood family $ \Sigma^{tSpan} $ (see \cite[Proposition 3.10]{BF14}). The fundamental group equipped with the thick Spanier subgroup topology is denoted by $\pi_1 ^{tSpan}(X, x_0)$. The thick Spanier subgroup of $X$, denoted by ${\Pi }^{sp}(X,x_0)$, is the intersection of all thick Spanier subgroups relative to open covers of $X$ \cite[Definition 3.8]{BF14}. Note that thick Spanier subgroups $\Pi^{sp}(\mathcal{U}, x_0)$ and so $\Pi^{sp}(X, x_0)$ are normal subgroups of $\pi_1(X, x_0)$.

Some researchers have attempted to compare the above topologies as follows.
It is proved in \cite{A20} that the lasso topology on the fundamental group $\pi_1(X, x_0)$ coincides with the Spanier subgroup topology.
Since every Spanier subgroup of the fundamental group $\pi_1(X, x_0)$ is also a path Spanier subgroup, then for any pointed space $ (X, x_0) $ the path Spanier topology on the fundamental group, $\pi_1^{pSpan}(X, x_0)$, is finer than the Spanier topology, $\pi_1^{Span}(X, x_0)$.
Also using  \cite[Proposition 3.16]{B13} the authors of \cite{A20} show that
if $X$ is a locally path connected space, then $\pi_1^{pSpan}(X, x_0)$ is coarser than $\pi_1^{\mathrm{\tau}}(X, x_0)$ (see \cite[Corollary 3.15]{A20}).
By considering the definitions of $\pi_{1}^{qtop}(X,x_0)$ and $\pi_1^{\mathrm{\tau}}(X, x_0)$ it is easy to see that $\pi_{1}^{qtop}(X,x_0)$ is finer than $\pi_1^{\mathrm{\tau}}(X, x_0)$.
Fischer and Zastrow \cite[Lemma 2.1]{FZ7} showed that the whisker topology is finer than the $ qtop $-topology on the universal path space $ \widetilde{X} $ for any space $ X $. Clearly, the result will hold for the fundamental group $\pi_{1}(X,x_{0})$ as a subspace of $ \widetilde{X} $.

By the above statements, one can summarize the relationship between the mentioned topologies on the fundamental group in the following chain 
(note that we use the symbol $\preccurlyeq$ to show the finer topology on a group. For example, $G^{\tau_1}\preccurlyeq  G^{\tau_2}$ means that $\tau_2$ is finer than $\tau_1$ and $G^{\tau_1}\prec  G^{\tau_2}$ means that $\tau_2$ is strictly finer than $\tau_1$). 
\[ \pi_1^{sh}(X, x_0) \preccurlyeq \pi_1^{tSpan}(X,x_0) \preccurlyeq \pi_1^{lasso} (X, x_0)=\pi_1^{Span}(X, x_0) \preccurlyeq \pi_1^{pSpan}(X, x_0)\]
\[ \preccurlyeq \pi_1^{\tau}(X, x_0) \preccurlyeq \pi_1^{qtop}(X, x_0)  \preccurlyeq \pi_1^{wh}(X, x_0).\]

Note that there are some examples to show that most of the above topologies are strictly finer than the previous one. As a good example, consider the {\it Hawaiian Earring} space, $\mathbb{HE}$. Using the results of \cite{A20, B13,F9,J} there exists the following chain of strictly finer topologies on the fundamental group of $\mathbb{HE}$: 
\[\pi_1^{Span}(\mathbb{HE}, 0) \prec \pi_1^{pSpan}(\mathbb{HE}, 0),\] 
\[\pi_1^{sh}(\mathbb{HE}, 0)\prec \pi_1^{\tau}(\mathbb{HE}, 0)\prec \pi_1^{qtop}(\mathbb{HE}, 0)\prec \pi_1^{wh}(\mathbb{HE}, 0).\]

For the generalized covering topology, it is proved in \cite[Proposition 3.24]{A20} that 
\[\pi_1^{\mathrm{qtop}}(X, x_0)  \preccurlyeq \pi_1^{\mathrm{gcov}}(X, x_0),\]
when $X$ is a connected and locally path connected. Note that $\pi_1^{wh}(\mathbb{HE}, 0)  \prec \pi_1^{gcov}(\mathbb{HE}, 0)$ and 
$\pi_1^{gcov}(\mathbb{HA}, b)  \prec \pi_1^{wh}(\mathbb{HA}, b)$, where $\mathbb{HA}$ is the {\it Harmonic Archipelago} and $b\neq 0$ (see \cite{A16,A20,FZ7}). Therefore, the whisker topology and the generalized covering topology can not be comparable, in general.
Moreover, if $X$ is locally path connected, paracompact and Hausdorff space , then by \cite[Theorem 7.6]{BF14} one can see that the equality $\pi_1^{sh}(X, x_0)=\pi_1^{tSpan}(X,x_0)=\pi_1^{Span}(X, x_0)$ holds. 

Let $(X,x_0)$ be a pointed topological space and H be a subgroup from $\pi _1(X,x_0)$, we define $\Sigma^H$ as follows:
\[\Sigma^H= \{K\leqslant \pi_1(X,x_0) \ | \ H\subseteq K\}\]
It is easy to see that $\Sigma^H$ is a neighbourhood family. Now we define $\pi_1^H(X,x_0)$ as subgroup topology on $\pi_1(X,x_0)$ determined by $\Sigma^H$.

In this paper, we are going to study most of the above topologies on $\pi _1(X,x_0)$ in view of some famous subgroups in the following chain of subgroups of the fundamental group ${\pi }_1(X,x_0)$ for a locally path connected space $(X,x_0)$ (see \cite{A16}).
\[\{e\}\leq {\pi }^s(X,x_0)\leq {\pi }^{sg}(X,x_0)\leq \pi ^{gc}(X,x_0)\leq {\widetilde{\pi }}^{sp}(X,x_0)\leq {\pi }^{sp}(X,x_0)\leq p_*\pi_1(\widetilde{X},\widetilde{x}), 
\]
where ${\pi }^s(X,x_0)$ is the subgroup of all small loops at $x_0$ \cite{V}, ${\pi }^{sg}(X,x_0)$ is the subgroup of all small generated loops \cite{V},
$\pi^{gc}(X,x_0)$ is the intersection of generalized covering subgroups \cite{A16}, ${\pi }^{sp}(X,x_0)$ is the Spanier group of $X$, the intersection of the Spanier subgroups relative to open covers of $X$ \cite[Definition 2.3]{FR},  ${\widetilde{\pi }}^{sp}(X,x_0)$ is the path Spanier group, i.e, the intersection of all path Spanier subgroups ${\widetilde{\pi }}(\mathcal{V},x_0)$ where $\mathcal{V}$ is a path open cover of $X$ \cite[Section 3]{T},    and $p_*\pi_1(\widetilde{X},\widetilde{x}_0)\cong \pi_1(\widetilde{X},\widetilde{x}_0) $ is the image of the induced homomorphism of a covering map $p:(\widetilde{X},\widetilde{x}_0)\rightarrow (X,x_0)$.

In Section 2, by reviewing subgroup topology on a group, we concentrate on a special kind of subgroup topology on a group $G$ with respect to a subgroup $H$ induced by the neighbourhood family $\Sigma^H= \{K\leqslant G \ | \ H\subseteq K\}$, denoted by $G^H$. Among presenting some properties of this kind of subgroup topology, we show that $G^H$ is a topological group if and only if $H$ is a normal subgroup of $G$. Moreover,
if $G^{\Sigma} = G^{\Sigma^{\prime}}$, then  $S_{\Sigma} = S_{\Sigma^{\prime}}$, where $S_{\Sigma}$ is the infinitesimal subgroup for the neighbourhood family $\Sigma$. Also the converse holds if $S_{\Sigma} \in \Sigma$ and $S_{\Sigma^{\prime}} \in \Sigma^{\prime}$.

In Section 3, we compare some of the well-known topologies on the fundamental group ${\pi }_1(X,x_0)$ with subgroup topologies with respect to the above famous subgroups of ${\pi }_1(X,x_0)$. Among giving some results, we prove that
if $H=\pi^{gc}(X,x_0)$, then $\pi_1^H(X,x_0)=\pi_1^{gcov}(X,x_0)$. Also, if $H=\pi^{sp}(X,x_0)$, then
$\pi_1^{Span}(X,x_0) \preccurlyeq \pi_1^H(X,x_0)$ and the equality holds if $X$ is coverable. Moreover, if $\pi^{sp}(X,x_0)$ is a semicovering subgroup, then 
$\pi_1^{Span}(X,x_0) \preccurlyeq \pi_1^H(X,x_0) \preccurlyeq \pi_1^{pSpan}(X,x_0).$
Also, $\pi_1^{pSpan}(X,x_0) \preccurlyeq \pi_1^{\widetilde{\pi}^{sp}(X,x_0)}(X,x_0)$ and the equality holds if $\widetilde{\pi}^{sp}(X,x_0)$ is a semicovering subgroup.  As a consequence 
 $\pi_1^{Span}(X,x_0) = \pi_1^{\pi^{sp}(X,x_0)}(X,x_0) = \pi_1^{gcov}(X,x_0)$ if X is  connected, locally path connected and semilocally $\pi^{gc}(X,x_0)$-connected space.  Also, $\pi_1^H(X,x_0) \preccurlyeq \pi_1^{qtop}(X,x_0) \preccurlyeq \pi_1^{wh}(X,x_0)$ if and only if X is semilocally H-connected at $x_0$. Moreover, $\pi_1^{wh}(X,x_0 \preccurlyeq \pi_1^{\pi^s(X,x_0)}(X,x_0)$ and if $X$ is semilocally $\pi^s(X,x_0)$-connected at $x_0$, then 
$\pi_1^{wh}(X,x_0) =\pi_1^{qtop}(X,x_0)= \pi_1^{\pi^s(X,x_0)}(X,x_0).$ At the end of this section, we sum up all results of this section in a diagram in order to compare various topologies on the fundamental group.
 
Finally, in Section 4, we prove that $\pi_1^{qtop}(X,x_0)$ is a topological group if X is semilocally $\pi^{gc}(X,x_0)$-connected or $[\pi_1(X,x_0):\pi^{gc}(X,x_0)]$ is finite. Also, we show that $\pi_1^H(X,x_0)$ is a topological group if and only if $H$ is a normal subgroup of $\pi_1(X,x_0)$. As a consequence, we have the following results.\\
$(i)$ If $\pi_1(X,x_0)$ is a Dedekind group, then $\pi_1^H(X,x_0)$ is a topological group, for every subgroup $H$ of $\pi_1(X,x_0)$.\\
$(ii)$ $\pi_1^{\Pi^{sp}(X,x_0)}(X,x_0)$, $\pi_1^{\pi^{sp}(X,x_0)}(X,x_0)$ and $\pi_1^{\pi^{sg}(X,x_0)}(X,x_0)$ are topological groups.
 

\section{Subgroup Topologies} \label{Sec2}

  A collection $\Sigma$ of subgroups of $G$ is called a \textit{neighbourhood family} if for any $H,K \in \Sigma$, there is a subgroup $S \in \Sigma$ such that $S \subseteq H\cap K$. As a result of this property, the collection of all left cosets of elements of $\Sigma$ forms a basis for a topology on $G$, which is called the subgroup topology determined by $\Sigma$ and we denote it by $G^{\Sigma}$. The \textit{subgroup topology} on a group $G$ specified by a neighbourhood family was defined in \cite[Section 2.5]{BS} and considered by some recent researchers such as \cite{A20,W}.
Since left translation by elements of $G$ determine self-homeomorphisms of $G$, they are homogeneous spaces. 
 Bogley et al. \cite{BS} focused on some general properties of subgroup topologies and by introducing the intersection $S_\Sigma=\cap \{H \ \vert \ H \in \Sigma\}$, called the \textit{infinitesimal} subgroup for the neighbourhood family $\Sigma$. They showed that the closure of the element $g\in G$ is the coset $g S_\Sigma$. 

Let H be a subgroup of a group $G$. Then we define $\Sigma^H$ as follows
\[\Sigma^H= \{K\leqslant G \ | \ H\subseteq K\}.\]
It is easy to see that $\Sigma^H$ is a neighbourhood family. We consider the subgroup topology on $G$ determined by $\Sigma^H$ and denote it by $G^H$. Note that the infinitesimal subgroup for the neighbourhood family $\Sigma^H$ is $H$. The following result can be obtained easily.
\begin{theorem}\label{2.1}
Let $H\leqslant G$ , then $G^H$ is discrete if and only if $H=1$. Also, $G^H$ is indiscrete if and only if $H=G$.
\end{theorem}

Using \cite[Theorem 2.9]{BS} and the above theorem we have the following result.
\begin{corollary}\label{2.2}
Let $H$ be a subgroup of $G$, then the following statements are equivalent.\\
$(i)$ $G^H$ is Hausdorff.
$(ii)$ $G^H$ is $T_0$.
$(iii)$ $G^H$ is totally disconnected.
$(iv)$ $G^H$ is discrete.
$(v)$ $H$ is the trivial subgroup.
\end{corollary} 

 It is pointed out in \cite{BS} that although the group $ G $ equipped with a subgroup topology may not necessarily a topological group, in general (it may not even a quasitopological group), because right translation maps by a fixed element of $G$ need not be continuous, but it has some of properties of topological groups  (for more details see Theorem 2.9 in \cite{BS}). Wilkins \cite[Lemma 5.4]{W} showed that a group G with the subgroup topology determined by a neighbourhood family $\Sigma$ is a topological group when all subgroups in $\Sigma$ are normal.
 Moreover, it is proved in \cite[Corollary 2.2]{A20} that a group equipped with a subgroup topology is a topological group if and only if all right translation maps are continuous.
A \textit{Dedekind group} is a group G such that every subgroup of G is normal. Clearly all abelian groups are Dedekind groups. A non-abelian Dedekind group is called a Hamiltonian group. It is proved that every Hamiltonian group is a direct product of the form $Q_8\times B\times D$, where $Q_8$ is the quaternion group,  $B$ is an elementary abelian $2$-group, and $D$ is a torsion abelian group with all elements of odd order (see \cite[p.190]{Hal}). Using these facts we have the following result. Let $H$ be a normal subgroup of $G$ such that the quotient group $G/H$ is a Dedekine group. Then $G^H$ is a topoplogical group.

 In the following theorem we vastly extend the above result.

\begin{theorem}\label{2.3}
Let $G$ be a group and $\Sigma$ be a neighbourhood family on $G$ such that $S_{\Sigma}\in \Sigma $. Then $G^{\Sigma}$ is a topological group if and only if $S_{\Sigma}$ is a normal subgroup of $G$. In particular, $G^H$ is a topological group if and only if $H$ is a normal subgroup of $G$.
\end{theorem}
\begin{proof}
Put $H=S_{\Sigma}$. Since $H\in \Sigma $, every subgroup $K$ in $\Sigma$ is a union of left cosets of $H$ and so $\{gH|g\in G\}$ is a basis for the subgroup topology of 
$G^{\Sigma}$. Suppose that $H$ is a normal subgroup of $G$, then $gH=Hg$ for all $g\in G$. Therefore any right translation map $r_x:G\rightarrow G$ is continuous since $r_x^{-1}(Hg)=Hgx^{-1}$. Hence by \cite[Proposition 2.1]{A20} $G^{\Sigma}$ is a topological group. 

Conversely, let $G^{\Sigma}$ be a topological group. Then any right translation map $r_x:G\rightarrow G$ is continuous. Therefore $Hx^{-1}=r_x^{-1}(H)$ is open in $G^{\Sigma}$. Since $\{gH|g\in G\}$ is a basis for the subgroup topology of $G^{\Sigma}$, there exists a left coset $gH$ such that $x^{-1}\in gH\subseteq Hx^{-1}$. Thus 
$x^{-1}H= gH\subseteq Hx^{-1}$ and so $x^{-1}Hx\subseteq H$. Hence $H$ is a normal subgroup of $G$.
\end{proof}
 
In the following we intend to compare two subgroup topologies.
\begin{theorem}\label{2.4}
Let $G$ be a group and $\Sigma$ and $\Sigma^{\prime}$ are two neighbourhood family on $G$ such that $G^{\Sigma} = G^{\Sigma^{\prime}}$. Then $S_{\Sigma} = S_{\Sigma^{\prime}}$ 
\end{theorem}

\begin{proof}
Let $H\in \Sigma$, then $H$ is open in $G^{\Sigma^{\prime}}$ since $G^{\Sigma} = G^{\Sigma^{\prime}}$. By definition of subgroup topology, $H$ is a union of some cosets of elements of $\Sigma^{\prime}$ say $H=\cup_{K\in \Sigma^{\prime}}gK$ for some $g\in G$. Since $H$ is a subgroup and it contains the trivial element, there is $gK$ such that $gK=K$ and so $K \subseteq H$. Therefore, since $S_{\Sigma^{\prime}} \subseteq K$ we have $S_{\Sigma^{\prime}} \subseteq H$ for each $H\in \Sigma$. Thus $S_{\Sigma^{\prime}} \subseteq S_{\Sigma}$. Similarly $S_{\Sigma} \subseteq S_{\Sigma^{\prime}}$, hence the result holds.
\end{proof}
 
 The following theorem shows that the converse of the above result holds under a condition. 
\begin{theorem}\label{2.5}
Let $G$ be a group and let $\Sigma$ and $\Sigma^{\prime}$ be two neighbourhood family on $G$ such that $S_{\Sigma} \leq S_{\Sigma^{\prime}}$ and $S_{\Sigma} \in \Sigma$. Then $G^{\Sigma^{\prime}} \preccurlyeq G^{\Sigma}$. In particular, if $H\leq K\leq G$, then $G^K \preccurlyeq G^H$.
\end{theorem}
\begin{proof}
Let $K \in \Sigma^{\prime}$. Then $S_{\Sigma} \subseteq K$ since $S_{\Sigma} = S_{\Sigma^{\prime}}$ and $S_{\Sigma^{\prime}} = \cap_{K\in \Sigma^{\prime}} K$. Thus $K$ is a union of some cosets of $S_{\Sigma}$ and so it is open in $G^{\Sigma}$ due to $S_{\Sigma} \in \Sigma$. Hence $G^{\Sigma^{\prime}} \preccurlyeq G^{\Sigma}$.    
\end{proof}


It is known that in every left (right) topological groups, any open subgroup is closed but the converse does not hold, in general. Note that by \cite[Proposition 2.4]{A20} one can show that the converse holds for subgroup topology if the infinitesimal subgroup is an open subgroup. Hence every closed subgroup of $G^H$ is also open.

 It is clear that if $H\in \Sigma$ for a neighbourhood family $\Sigma$, then $H$ is an open subgroup in $G^{\Sigma}$. Note that the converse does not hold in general. This property seems essential when we intend to use the topology on the group to classifying its subgroups.  The subgroup topology on a group $G$ determined by the neighbourhood family $\Sigma$ is called \textit{regular} if every open subgroup in $G$ belongs to $\Sigma$. Note that the subgroup topology on $G^H$ is regular. The following theorem seems interesting in order to dealing with regularity.
 
\begin{proposition}\label{2.6}
The subgroup topology on a group $G$ determined by neighbourhood family $\Sigma$ is regular if and only if each subgroup $H$ of $G$ which contains a subgroup $K$ with $K \in \Sigma$, belongs also in $\Sigma$.
\end{proposition}

In the following theorem we consider the subgroup topology on the direct product of groups. Note that by a routine calculation one can show that $G^{\Sigma}\times G'^{\Sigma '} \cong (G\times G')^{\Sigma\times \Sigma '}$.
\begin{theorem}\label{2.7}
Let G and $G^\prime$ are two group, H and $H^\prime$ are subgroups of G and $G^\prime$ respectively, then we have $G^H\times {G^\prime}^{H^\prime} \cong (G\times G^\prime)^{H\times H^\prime}$.
\end{theorem}

\begin{proof}
Let $H\times H^\prime \subseteq M\times N \leq G\times G^\prime$ i.e., $M\times N$ is open in $G^H\times {G^\prime}^{H^\prime}$, then clearly $M\times N$ is open in $(G\times G^\prime)^{H\times H^\prime}$.
Conversely, let $H\times H^\prime \subseteq L \leq G\times G^\prime$ i.e., $L$ is open in $(G\times G^\prime)^{H\times H^\prime}$, then $L$ is a union of some cosets of $H\times H^\prime$. Hence $L$ is open in $G^H\times {G^\prime}^{H^\prime}$.
\end{proof}


\section{Comparison of Topologies on Fundamental Groups}
We recall a semicovering map $p : (\widetilde{X},\widetilde{x}_0) \rightarrow (X,x_0)$, introduced by Brazas \cite{B12,B14}, which is a local homeomorphism with continuous lifting of paths and homotopies. 
Also a subgroup $H$ of the fundamental group $\pi_1(X,x_0)$ is called a \textit{semicovering subgroup} if there is a semicovering map $p : (\widetilde{X},\widetilde{x}_0) \rightarrow (X,x_0)$ such that $H = p_\ast \pi_1(\widetilde{X},\widetilde{x}_0)$ (see \cite{B14, KTM}).

Pakdaman et al. \cite[Definition 2.4]{PTM17} called a pointed topological space $(X,x_0)$ \textit{coverable} if it has the categorical universal covering space or equivalently the Spanier group, $\pi^{sp}(X,x_0)$, is a covering subgroup.

In the following theorem, we compare the Spanier topology on the fundamental group with the subgroup topology induced by the Spanier subgroup.
\begin{theorem}\label{3.1}
$(i)$ If $H=\pi^{sp}(X,x_0)$, then $\pi_1^{Span}(X,x_0) \preccurlyeq \pi_1^H(X,x_0)$.\\
$(ii)$ $\pi_1^{Span}(X,x_0)=\pi_1^{\pi^{sp}(X,x_0)}(X,x_0)$ if and only if $X$ is coverable.\\
$(iii)$ If $\pi_1^{Span}(X,x_0)=\pi_1^K(X,x_0)$ for a subgroup $K$ of $\pi_1(X,x_0)$, then $K=\pi^{sp}(X,x_0)$.
\end{theorem}
\begin{proof}
$(i)$ It is known that $\pi_1^{Span}(X,x_0)$ has subgroup topology with respect to $\Sigma^{Span}=\{K\leqslant \pi_1(X,x_0) \ | \ K \text{ is a Spanier subgroup} \}$ (see \cite{W}).
Since $H=\pi^{sp}(X,x_0)$,  $\Sigma^H=\{K\leqslant \pi_1(X,x_0) \ | \ \pi^{sp}(X,x_0)\subseteq K \}$,
and $\pi^{sp}(X,x_0)=\cap_{K\in \Sigma^{sp}} K$, we have
$\Sigma^{sp}\subseteq \Sigma^H$. Hence $\pi_1^{span}(X,x_0)\preccurlyeq \pi_1^H(X,x_0)$.\\
$(ii)$ By definition $X$ be coverable if and only if $\pi^{sp}(X,x_0)$ is a covering subgroup. By \cite[Theorem 2.5.13]{Span} $\pi^{sp}(X,x_0)$ is a covering subgroup if and only if then $\pi^{sp}(X,x_0)\in \Sigma^{Span}$. Hence  $X$ be coverable if and only if $\pi_1^H(X,x_0)=\pi_1^{Span}(X,x_0)$.\\
$(iii)$ It holds by Theorem \ref{2.4}.
\end{proof}

Note that in the above theorem the strict inequality holds for the Hawaiian Earring $\mathbb{HE}$ since it is not coverable (see \cite[Theorem 2.8]{PTM17}). 

In the following theorem, we compare the path Spanier topologies on the fundamental group with the subgroup topologies induced by the Spanier and the path Spanier subgroups.
\begin{theorem}\label{3.2}
$(i)$ $\pi_1^{pSpan}(X,x_0) \preccurlyeq \pi_1^{\widetilde{\pi}^{sp}(X,x_0)}(X,x_0)$.\\
$(ii)$ If $\pi_1^{pSpan}(X,x_0)=\pi_1^K(X,x_0)$ for a subgroup $K$ of $\pi_1(X,x_0)$, then $K=\widetilde{\pi}^{sp}(X,x_0)$.

Let $X$ be connected and locally path connected, then\\
$(iii)$ $\pi_1^{pSpan}(X,x_0)=\pi_1^{\widetilde{\pi}^{sp}(X,x_0)}(X,x_0)$ if and only if $\widetilde{\pi}^{sp}(X,x_0)$ is a semicovering subgroup.\\
$(iv)$ $\pi_1^{\pi^{sp}(X,x_0)}(X,x_0) \preccurlyeq \pi_1^{pSpan}(X,x_0)$ if and only if $\pi^{sp}(X,x_0)$ is a semicovering subgroup.\\
$(v)$ If $\pi^{sp}(X,x_0)$ is a semicovering subgroup, then $\pi_1^{pSpan}(X,x_0)=\pi_1^{\pi^{sp}(X,x_0)}(X,x_0)$ if and only if $\pi^{sp}(X,x_0)=\widetilde{\pi}^{sp}(X,x_0)$.
\end{theorem}
\begin{proof}
$(i)$ The neighbourhood family of $\pi_1^{pSpan}(X,x_0)$ and $\pi_1^{\widetilde{\pi}^{sp}(X,x_0)}(X,x_0)$ are $\Sigma^{pSpan}=\{K \ | \ \text{K is a path Spanier subgroup} \}$ and $\Sigma^{\widetilde{\pi}^{sp}(X,x_0) }=\{H \ | \ \widetilde{\pi}^{sp}(X,x_0)\subseteq H \}$, respectively. By definition of $\widetilde{\pi}^{sp}(X,x_0)$,  $\Sigma^{pSpan}\subseteq \Sigma^{\widetilde{\pi}^{sp}(X,x_0) }$.
Thus $\pi_1^{pSpan}(X,x_0)\preccurlyeq \pi_1^{\widetilde{\pi}^{sp}(X,x_0)}(X,x_0)$.\\
$(ii)$ It holds by Theorem \ref{2.4}.\\
$(iii)$ By \cite[Theorem 4.1]{T}  $\widetilde{\pi}^{sp}(X,x_0)$ is a semicovering subgroup if and only if $\widetilde{\pi}^{sp}(X,x_0)\in \Sigma^{pSpan}$.
Hence $\pi_1^{\widetilde{\pi}^{sp}(X,x_0)}(X,x_0)= \pi_1^{pSpan}(X,x_0)$ if and only if $\widetilde{\pi}^{sp}(X,x_0)$ is a semicovering subgroup.\\
$(iv)$ By \cite[Theorem 4.1]{T} $\pi^{sp}(X,x_0)$ is a semicovering subgroup if and only if  there is a path open cover $\mathcal{U}$ such that $\widetilde{\pi}(\mathcal{U},x_0)\leq \pi^{sp}(X,x_0)$. If $\pi^{sp}(X,x_0)$ is a semicovering subgroup, then $\pi^{sp}(X,x_0)$ is open in $\pi_1^{pSpan}(X,x_0)$. If $\pi^{sp}(X,x_0)\subseteq K\leqslant \pi_1(X,x_0)$ i.e., $K$ is open in $\pi_1^H(X,x_0)$, then $K=\bigcup _{[\alpha]\in K}[\alpha]\pi^{sp}(X,x_0)$. Since $[\alpha]\pi^{sp}(X,x_0)$ is open in $\pi_1^{pSpan}(X,x_0)$, so is $K$ and hence $\pi_1^H(X,x_0) \preccurlyeq \pi_1^{pSpan}(X,x_0)$.

Conversely, if $\pi_1^{\pi^{sp}(X,x_0)}(X,x_0) \preccurlyeq \pi_1^{pSpan}(X,x_0)$, then $\pi^{sp}(X,x_0)$ is open in $\pi_1^{pSpan}(X,x_0)$. Therefore $\pi^{sp}(X,x_0)$ is a union of some cosets of some $\widetilde{\pi}(\mathcal{U},x_0)$. Since the identity element does not belong to any cosets except the subgroups, one of the cosets of some $\widetilde{\pi}(\mathcal{U},x_0)$ must be a subgroup. Hence $\pi^{sp}(X,x_0)$ contains one of the $\widetilde{\pi}(\mathcal{U},x_0)$ and so $\pi^{sp}(X,x_0)$ is a semicovering subgroup.\\
$(v)$ By $(ii)$ and $(iv)$.
\end{proof}

\begin{remark}\label{3.3}
\begin{itemize}
\item[(i)] Note that using  \cite[Proposition 3.16]{B13} the authors of \cite{A20} show that if $X$ is a locally path connected space, then $\pi_1^{pSpan}(X, x_0)$ is coarser than $\pi_1^{\mathrm{\tau}}(X, x_0)$ (see \cite[Corollary 3.15]{A20}).
\item[(ii)] By considering the definitions of $\pi_{1}^{qtop}(X,x_0)$ and $\pi_1^{\tau}(X, x_0)$ it is easy to see that $\pi_{1}^{qtop}(X,x_0)$ is finer than $\pi_1^{\tau}(X, x_0)$ (see \cite{B13}). The strict inequality holds for the Hawaiian Earring $\mathbb{HE}$, $\pi_1^{qtop}(\mathbb{HE},0)$ is not a topological group by \cite{F11} and  $\pi_1^{\tau}(\mathbb{HE},0)$  is a topological group by \cite{B13}. The equality holds if and only if $\pi_1^{qtop}(X,x_0)$ is a topological group (see \cite{B13}).
\item[(iii)]  In \cite[Proposition 3.2]{BF24} it is shown that the shape topology of $\pi_1^{sh}(X, x_0)$ is coarser than that of $\pi_1^{\tau}(X, x_0)$. Note that the strict inequality holds for the Hawaiian Earring $\mathbb{HE}$ i.e, $\pi_1^{sh}(\mathbb{HE}, 0)\prec \pi_1^{\tau}(\mathbb{HE}, 0)$ (see \cite[Example 3.26]{B13}). 
Note that by \cite[Proposition 3.5]{BF14} one can show that $\pi_1^{Span}(X, x_0)$ is finer than $\pi_1^{tSpan}(X, x_0)$. Also by \cite[Proposition 5.8]{BF14} $\pi_1^{tSpan}(X, x_0)$ is finer than $\pi_1^{sh}(X, x_0)$. Moreover, by \cite[Theorem 7.6]{BF14} one can easily see that the equality 
$\pi_1^{sh}(X, x_0)=\pi_1^{tSpan}(X, x_0)= \pi_1^{Span}(X,x_0)$ holds if $X$ is locally path connected, paracompact and Hausdorff.
\end{itemize}
\end{remark}

In the following theorem, we compare the generalized covering topology on the fundamental group with the subgroup topology induced by the generalized covering subgroup.
\begin{theorem}\label{3.4}
Let $H=\pi^{gc}(X,x_0)$, then $\pi_1^H(X,x_0)=\pi_1^{gcov}(X,x_0)$. Also, if $\pi_1^{gcov}(X,x_0)=\pi_1^K(X,x_0)$ for a subgroup $K$ of $\pi_1(X,x_0)$, then $K=\pi^{gc}(X,x_0)$.
\end{theorem}
\begin{proof}
By definitions of  $\pi_1^{gcov}(X,x_0)$  and $\pi^{gc}(X,x_0)$ it is easy to see that $\Sigma^{gcov}\subseteq \Sigma^H$ and so $\pi_1^{gcov}(X,x_0)\preccurlyeq \pi_1^H(X,x_0)$.

Let $\pi^{gc}(X,x_0)\subseteq K\leqslant \pi_1(X,x_0)$ i.e., $K$ is open in $\pi_1^H(X,x_0)$, then $K$ is a union of some cosets of $\pi^{gc}(X,x_0)$. By \cite{B15} $\pi^{gc}(X,x_0)$ is a generalizes covering subgroup and so it is open in $\pi_1^{gcov}(X,x_0)$.  Hence $K$ is open in $\pi_1^{gcov}(X,x_0)$ and thus $\pi_1^H(X,x_0)\preccurlyeq \pi_1^{gcov}(X,x_0)$.
Therefore $\pi_1^H(X,x_0)=\pi_1^{gcov}(X,x_0)$ when $H=\pi^{gc}(X,x_0)$.

If $\pi_1^{gcov}(X,x_0)=\pi_1^K(X,x_0)$ for a subgroup $K$ of $\pi_1(X,x_0)$, then by Theorem \ref{2.4} $K=\pi^{gc}(X,x_0)$.
\end{proof}

The following result is a consequence Theorems \ref{2.4}, \ref{2.5} and \ref{3.4}.
\begin{corollary}\label{3.5}
$(i)$ $\pi_1^{\widetilde{\pi}^{sp}(X,x_0)}(X,x_0) \preccurlyeq \pi_1^{gcov}(X,x_0)$. Moreover, $\pi_1^{\widetilde{\pi}^{sp}(X,x_0)}(X,x_0)= \pi_1^{gcov}(X,x_0)$ if and only if $\pi^{gc}(X,x_0)=\widetilde{\pi}^{sp}(X,x_0)$.\\
$(ii)$ $\pi_1^{gcov}(X,x_0) \preccurlyeq \pi_1^{\pi^{sg}(X,x_0)}(X,x_0)$. Moreover, $\pi_1^{gcov}(X,x_0)= \pi_1^{\pi^{sg}(X,x_0)}(X,x_0)$ if and only if $\pi^{gc}(X,x_0)=\pi^{sg}(X,x_0)$.
\end{corollary}

Note that as an example for the equality in the above corollary consider $\pi^{sg}(\mathbb{HA},b)=\pi^{gc}(\mathbb{HA},b)$ (\cite [Example 3.12]{A16}). Also for the strict inequality consider the space $RX$ in Example \ref{4.2} since $\pi^{sg}(RX,x_0)=1 \neq \pi^{gc}(RX,x_0)$.\\

A topological space $X$ is called \textit{semilocally $H$-connected at } $x_{0}\in{X}$ if there exists an open neighbourhood $U\ of\ x_{0}$ with $i_*{\pi }_1(U,x_{0})\le H$, where $i_*=\pi_1(i)$ is the induced homomorphism by the inclusion $i:U\hookrightarrow X$. Also, $X$ is called \textit{semilocally $H$-connected} if for every $x\in X$ and for every path $\alpha$ from $x_0$ to $x$ the space $X$ is semilocally $[\alpha^{-1}H\alpha]$-connected at $x$ (see \cite[Definition 4.1]{A16}).
If X is semilocally $H$-connected for $H=\pi^{gc}(X,x_0)$, then $X$ is coverable and $\pi^{gc}(X,x_0)=\pi^{sp}(X,x_0)$. Therefore by the above theorems we have the following corollary (see \cite[Corollary 4.8]{A16}).

\begin{corollary}\label{3.6}
Let $H=\pi^{gc}(X,x_0)$ and $X$ be a connected, locally path connected and semilocally $H$-connected space, then $\pi_1^{Span}(X,x_0) = \pi_1^{\pi^{sp}(X,x_0)}(X,x_0) = \pi_1^{gcov}(X,x_0)$.
\end{corollary}

\begin{remark}\label{3.7}
Fischer and Zastrow \cite[Lemma 2.1]{FZ7} showed that the whisker topology is finer than the compact-open quotient topology on the universal path space $ \widetilde{X} $ for any space $ X $. Clearly, the result will hold for the fundamental group $\pi_{1}(X,x_{0})$ as a subspace of $ \widetilde{X} $ i.e $\pi_1^{qtop}(X,x_0) \preccurlyeq \pi_1^{wh}(X,x_0)$. If $X$  is locally path connected, then the equality holds if and only if $X$ is SLT at $x_0$ (see \cite[Corollary 3.3]{ PA17}). The strict inequality holds for the Hawaiian Earring $\mathbb{HE}$ (see \cite[Example 3.25]{A20}).\\
\end{remark}



In the following theorem, we compare the whisker and compact-open quotient topologies on the fundamental group with the subgroup topology induced by a subgroup $H$ under a semilocally connectedness condition.

\begin{theorem}\label{3.8}
Let $(X,x_0)$ be a pointed topological space and $H\leqslant \pi_1(X,x_0)$. Then $X$ is semilocally $H$-connected at $x_0$ if and only if $\pi_1^H(X,x_0) \preccurlyeq \pi_1^{qtop}(X,x_0) \preccurlyeq \pi_1^{wh}(X,x_0)$.
\end{theorem}
\begin{proof}
Since $X$ is semilocally $H$-connected at $x_0$, there is an open neighbourhood $U$ of $x_0$ such that
\[i_*\pi_1(U,x_0)\subseteq H\]
Thus for any $K\in \Sigma^H$ we have $i_\ast (U,x_0)\subseteq K$.
By \cite[Theorem 4.2]{A16} $K$ is open in $\pi_1^{wh}(X,x_0)$. Hence $\pi_1^H(X,x_0)\preccurlyeq \pi_1^{wh}(X,x_0)$.

Since $X$ is connected, locally path connected and semilocally H-connected, by \cite[Corollary 4.9]{A16} X is semilocally path $H$-connected and so by \cite[Proposition 4.3]{A16} $H$ is open in $\pi_1^{qtop}(X,x_0)$. Therefore $K$ is open in $\pi_1^{qtop}(X,x_0)$ for all $K\in \Sigma^H$ and hence $\pi_1^H(X,x_0)\preccurlyeq \pi_1^{qtop}(X,x_0)$.

Conversely, since $\pi_1^H(X,x_0)\preccurlyeq \pi_1^{wh}(X,x_0)$, $H$ is open in $\pi_1^{wh}(X,x_0)$ and by \cite[Theorem 4.2]{A16} $X$ is semilocally $H$-connected at $x_0$.
\end{proof}

It is known that a connected, locally path connected space $X$ is semilocally $H$-connected if and only if $H$ is a covering subgroup of $\pi_1(X,x_0)$ (see \cite[Proposition 4.6]{A16}). Therefore, the following corollary is a consequence of this fact and Theorem \ref{3.8}.

\begin{corollary}\label{3.9}
Let X be a connected, locally path connected space and H be a subgroup of $\pi_1(X,x_0)$, then $\pi_1^H(X,x_0) \preccurlyeq \pi_1^{qtop}(X,x_0) \preccurlyeq \pi_1^{wh}(X,x_0)$ if and only if H is a covering subgroup.
 \end{corollary}
 
Let $COV( X )$ and $GCOV( X )$ denote the category of all coverings and generalized coverings of $X$, respectively. Then for a connected, locally path connected space $X$, the categorical equality $GCOV( X ) = COV( X )$ holds if and only if $X$ is semilocally $\pi^{gc}(X,x_0)$-connected (see \cite[Corollary 4.4]{A16}). Therefore the following corollary is a consequence of these facts and Corollary \ref{3.9}.

\begin{corollary}\label{3.10}
 Let $X$ be a connected and locally path connected space. Then\\
$(i)$ $\pi_1^{gcov}(X,x_0)=\pi_1^{qtop}(X,x_0) \preccurlyeq \pi_1^{wh}(X,x_0)$ if and only if $X$ is semilocally $\pi^{gc}(X,x_0)$-connected.\\
$(ii)$ $\pi_1^{gcov}(X,x_0)=\pi_1^{qtop}(X,x_0) \preccurlyeq \pi_1^{wh}(X,x_0)$ if and only if $GCOV (X) = COV (X)$.
 \end{corollary}

It is known that $\pi_1^{wh}(X,x_0)$ is a subgroup topology with respect to $\Sigma^{wh} =$\\ 
$ \{i_\ast \pi_1(U,x_0) | \text{U is an open neighborhood of X at}\ x_0\}$ (see \cite[p. 945]{A20}).
It implies from \cite[Proposition 3.8]{A16} that the infinitesimal subgroup of $\pi_1^{wh}(X,x_0)$ is $\pi^s(X,x_0)$.  By \cite[Proposition 3.20]{A20}, $X$ is semilocally $\pi^s(X,x_0)$-connected at $x_0$ if and only if $\pi^s(X,x_0)$ is open in $\pi_1^{wh}(X,x_0)$. Therefore we have the following result. 

\begin{corollary}\label{3.11}
Let $(X,x_0)$ be a pointed topological space, then 
\[\pi_1^{wh}(X,x_0 \preccurlyeq \pi_1^{\pi^s(X,x_0)}(X,x_0).\]
Also, $X$ is semilocally $\pi^s(X,x_0)$-connected at $x_0$ if and only if
\[\pi_1^{qtop}(X,x_0) =\pi_1^{wh}(X,x_0)= \pi_1^{\pi^s(X,x_0)}(X,x_0).\]
Moreover, if $\pi_1^{wh}(X,x_0)=\pi_1^K(X,x_0)$ for a subgroup $K$ of $\pi_1(X,x_0)$, then $K=\pi^{s}(X,x_0)$.
\end{corollary}

The following example shows that the condition semilocally $\pi^s(X,x_0)$-connectedness cannot be omitted for the equality in Corollary \ref{3.11}.
 
\begin{example}\label{3.12}
Let $\mathbb{HE}$ be the Hawaiian earring space, then by \cite{V} $\pi^s(\mathbb{HE},0) = 1$. Therefore by Theorem \ref{2.1} $\pi_1^{\pi^s(\mathbb{HE},0)}(\mathbb{HE},0)$ is discrete. By \cite[Remark 3.19]{A20} $\pi_1^{wh}(\mathbb{HE},0)$ is not discrete. Hence $ \pi_1^{wh}(\mathbb{HE},0) \prec \pi_1^{\pi^s(\mathbb{HE},0)}(\mathbb{HE},0)$.
\end{example}

\begin{remark}\label{3.13} $(i)$ An Alexandroff space is a topological space such that for every $x\in X$ there exists an open neighbourhood $W_x$ such that for each open neighbourhood $U$ of $x$ we have $W_X\subseteq U$. Let $X$ be an Alexandroff space and $W_{x_0}$ be the smallest open neighbourhood of $x_0$. Then by Corollary \ref{3.11} $\pi_1^H(X,x_0)=\pi_1^{wh}(X,x_0)$, where $H=i_*\pi_1(W_{x_0},x_0)=\pi^{s}(X,x_0)$ .\\
$(ii)$ It is known that every nonempty open or closed subset of $\pi_1^{qtop}(X,x_0)$ is a disjoint union of some cosets of $\pi^{sg}(X,x_0)$ (see \cite{TPM}). Therefore if $H=\pi^{sg}(X,x_0)$, then $\pi_1^{qtop}(X,x_0) \preccurlyeq \pi_1^H(X,x_0)$. Note that by Theorem \ref{3.8} if $X$ is semilocally $\pi^{sg}(X,x_0)$-connected at $x_0$, then $\pi_1^{qtop}(X,x_0)= \pi_1^{\pi^{sg}(X,x_0)}(X,x_0)$.\\
$(iii)$ By Theorem \ref{2.5} and $\pi^s(X,x_0) \leq \pi^{sg}(X,x_0)$ we have $\pi_1^{\pi^{sg}(X,x_0)}(X,x_0) \preccurlyeq $\\ $\pi_1^{\pi^s(X,x_0)}(X,x_0)$. The equality holds if and only if $\pi^s(X,x_0)= \pi^{sg}(X,x_0)$ by Theorem \ref{2.4}. The strict inequality holds for $\mathbb{HA}$ at $b\neq 0$ since $\pi^{s}(\mathbb{HA},b)=1$ and $\pi^{sg}(\mathbb{HA},b)=\pi_1(\mathbb{HA},b)$ (\cite [Example 3.12]{A16}).\\
$(iv)$ Consider $\mathbb{HE}$ the Hawaiian Earring space and $\mathbb{HA}$ the Harmonic Archipelago space. Then by \cite[Example 3.25]{A20} $\pi_1^{wh}(\mathbb{HE},0) \preccurlyeq \pi_1^{gcov}(\mathbb{HE},0)$ and by \cite[Example 3.26]{A20} $\pi_1^{gcov}(\mathbb{HA},b) \preccurlyeq \pi_1^{wh}(\mathbb{HA},b)$ where $b \in \mathbb{HA}$ is a noncanonical based point. Therefore, the whisker topology and the generalized covering topology are not comparable, in general.\\
$(v)$ It is proved in \cite[Proposition 3.24]{A20} that if $X$ is connected and locally path connected, then 
$\pi_1^{qtop}(X, x_0)  \preccurlyeq \pi_1^{gcov}(X, x_0)$.
 The strict inequality holds for the Hawaiian Earring $\mathbb{HE}$ (see \cite[Example 3.25]{A20}). The equality holds for any semilocally  $\pi^{gc}(X,x_0)$-connected space (see Theorem \ref{3.8}).\\
\end{remark}

Finally, we summarize all results of this section in order to compare various topologies on the fundamental group in the following diagram
(note that $A \longrightarrow B$ means that $A\preccurlyeq B$).
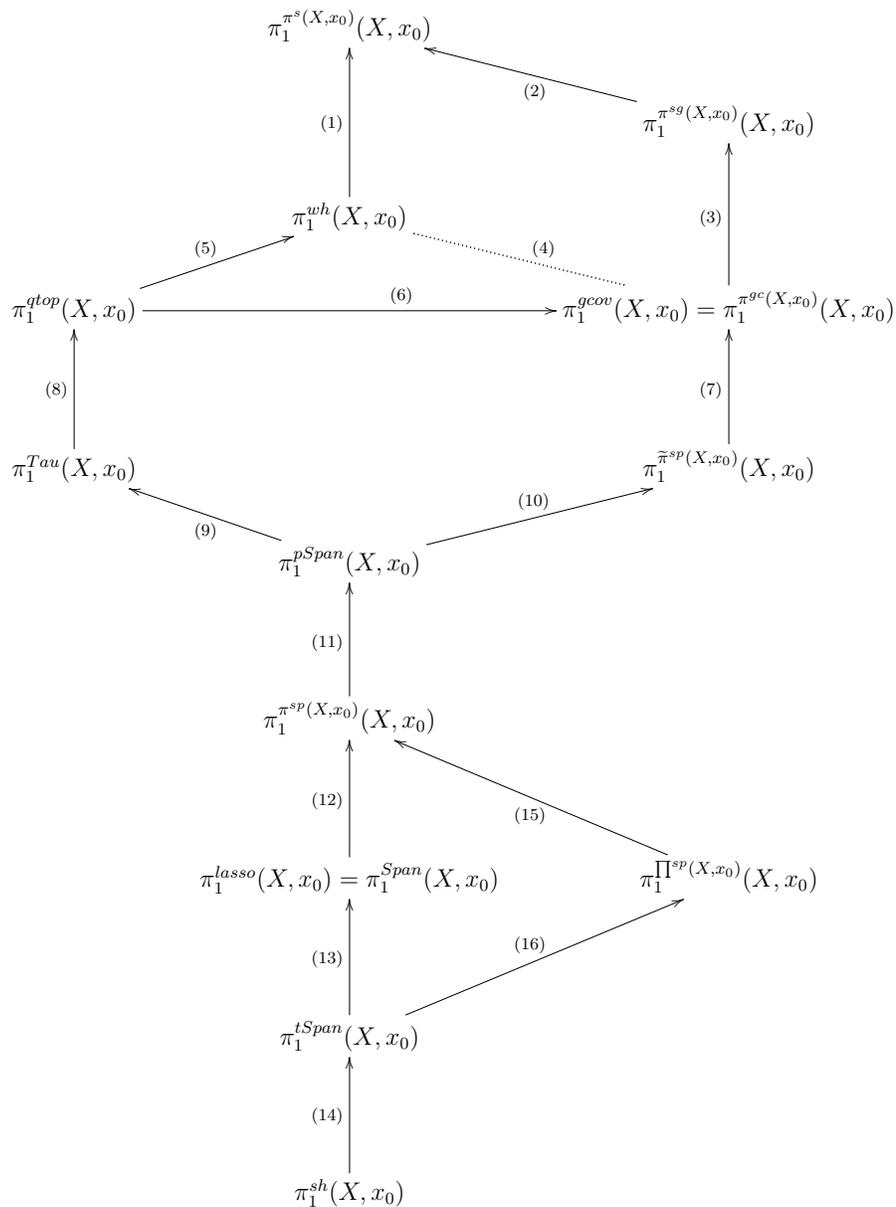
\[ \scalebox{0.8} { \xymatrix{
&\pi_1^{\pi^s(X,x_0)}(X,x_0)&\\
&&\pi_1^{\pi^{sg}(X,x_0)}(X,x_0) \ar[lu]^{(2)}\\
&\pi_1^{wh}(X,x_0) \ar[uu]^{(1)} \ar@{.}[rd]^{(4)}&\\
\pi_1^{qtop}(X,x_0) \ar[ru]^{(5)} \ar[rr]^{(6)}&&\pi_1^{gcov}(X,x_0)=\pi_1^{\pi^{gc}(X,x_0)}(X,x_0) \ar[uu]^{(3)} \ar@{.}[lu] \\
&&&\\
\pi_1^{Tau}(X,x_0) \ar[uu]^{(8)}&&\pi_1^{\widetilde{\pi}^{sp}(X,x_0)}(X,x_0) \ar[uu]^{(7)}\\
&\pi_1^{pSpan}(X,x_0) \ar[lu]^{(9)} \ar[ru]^{(10)}&\\
&&&\\
&\pi_1^{\pi^{sp}(X,x_0)}(X,x_0) \ar[uu]^{(11)}&\\
&&&\\
&\pi_1^{lasso}(X,x_0)=\pi_1^{Span}(X,x_0)\ar[uu]^{(12)}&\pi_1^{\prod^{sp}(X,x_0)}(X,x_0)\ar[luu]^{(15)}\\
&&&\\
&\pi_1^{tSpan}(X,x_0) \ar[uu]^{(13)}\ar[ruu]^{(16)}&\\
&&&\\
&\pi_1^{sh}(X,x_0)\ar[uu]^{(14)}&\\
 }}\]
In the following, according to the enumeration in the above diagram, we give references and complementary notes for each arrow.
\begin{enumerate}
 \item[(1)] See Corollary \ref{3.11}. The equality holds if and only if $X$ is semilocally $\pi^s(X,x_0)$-connected at $x_0$ (see \cite[Proposition 4.21]{Br12}). Since $\mathbb{HE}$ is not semilocally simply connected at $0$ and $\pi^{s}(\mathbb{HE},0)=1$, the strict inequality holds for $\mathbb{HE}$ at $0$ (see Example \ref{3.12}).
\item[(2)] By Theorem \ref{2.5} and $\pi^s(X,x_0) \leq \pi^{sg}(X,x_0)$. The equality holds if and only if $\pi^s(X,x_0)= \pi^{sg}(X,x_0)$ by Theorem \ref{2.4}. The strict inequality holds for $\mathbb{HA}$ at $b\neq 0$ since $\pi^{s}(\mathbb{HA},b)=1$ and $\pi^{sg}(\mathbb{HA},b)=\pi_1(\mathbb{HA},b)$ (\cite [Example 3.12]{A16}).
\item[(3)] By Theorem \ref{2.5} and $\pi^{sg}(X,x_0) \leq \pi^{gc}(X,x_0)$. The equality holds if and only if $\pi^{sg}(X,x_0)= \pi^{gc}(X,x_0)$ by Theorem \ref{2.4}. As an example $\pi^{sg}(\mathbb{HA},b)=\pi^{gc}(\mathbb{HA},b)$ (\cite [Example 3.12]{A16}).
The strict inequality holds for the space $RX$ in Example \ref{4.2} since $\pi^{sg}(RX,x_0)=1 \neq \pi^{gc}(RX,x_0)$.
\item[(4)] By \cite[Example 3.25]{A20} $\pi_1^{wh}(\mathbb{HE},0) \prec \pi_1^{gcov}(\mathbb{HE},0)$ and by \cite[Example 3.26]{A20} $\pi_1^{gcov}(\mathbb{HA},b) \prec \pi_1^{wh}(\mathbb{HA},b)$, where $b \in \mathbb{HA}$ is a non canonical based point. Thus $\pi_1^{wh}(X,x_0)$ and $\pi_1^{gcov}(X,x_0)$ are not comparable, in general. By Corollary \ref{3.10} if $X$ is connected and locally path connected, then $\pi_1^{gcov}(X,x_0) \preccurlyeq \pi_1^{wh}(X,x_0)$ if and only if $X$ is semilocally $\pi^{gc}(X,x_0)$-connected. The equality holds for any semilocally simply connected space.
\item[(5)] By \cite[Lemma 2.1]{FZ7}. If $X$  is locally path connected, then the equality holds if and only if $X$ is SLT at $x_0$ (see \cite[Corollary 3.3]{ PA17}). The strict inequality holds for $\mathbb{HE}$ (see \cite[Example 3.25]{A20}).
\item[(6)] See \cite[Proposition 3.24]{A20}. Note that $X$  is connected and locally path connected. The strict inequality holds for $\mathbb{HE}$ (see \cite[Example 3.25]{A20}). The equality holds if and only if $X$ is a semilocally  $\pi^{gc}(X,x_0)$-connected space (see Corollary \ref{3.10}).
\item[(7)] By Theorem \ref{3.4} and $\pi^{gc}(X,x_0) \leq \widetilde{\pi}^{sp}(X,x_0)$. By Theorems \ref{2.4} and \ref{2.5}, the equality holds if and only if $\pi^{gc}(X,x_0)= \widetilde{\pi}^{sp}(X,x_0)$.
\item[(8)] See \cite{B13}. The strict inequality holds for $\mathbb{HE}$ since $\pi_1^{qtop}(\mathbb{HE},0)$ is not a topological group by \cite{F11} and  $\pi_1^{\tau}(\mathbb{HE},0)$  is a topological group by \cite{B13}. The equality holds if and only if $\pi_1^{qtop}(X,x_0)$ is a topological group (see \cite{B13}).
\item[(9)] See \cite[Corollary 3.15]{A20}. Note that $X$  is locally path connected. By \cite[Theorem 3.16]{A20} the equality holds if $X$ is locally path connected and semilocally small generated.
\item[(10)] By Theorem \ref{3.2}. The equality holds if and only if $ \widetilde{\pi}^{sp}(X,x_0)$ is a semicovering subgroup (see Theorem \ref{3.2}). 
\item[(11)] By Theorem \ref{3.2} if $\pi^{sp}(X,x_0)$ is a semicovering subgroup. By Theorems \ref{2.4} and \ref{2.5}, the equality holds if and only if $\pi^{sp}(X,x_0)= \widetilde{\pi}^{sp}(X,x_0)$.
\item[(12)] By \cite[Theorem 3.5]{A20} and Theorem \ref{3.1}. Also, by Theorem \ref{3.1}, the equality holds if and only if $ \pi^{sp}(X,x_0)$ is a covering subgroup or equivalently $X$ is coverable. The strict inequality holds for $\mathbb{HE}$ since it is not coverable (see \cite[Theorem 2.8]{PTM17}). 
\item[(13)] By \cite[Proposition 3.5]{BF14}. The equality holds if $X$ is locally path connected, paracompact and Hausdorff (see  \cite[Theorem 7.6]{BF14}).
\item[(14)] By \cite[Proposition 5.8]{BF14}. The equality holds if $X$ is locally path connected, paracompact and Hausdorff (see  \cite[Theorem 7.6]{BF14}).
\item[(15)] By \cite[Proposition 3.5]{BF14} $\pi^{sp}(X,x_0)\leq \Pi^{sp}(X,x_0)$. The equality holds if $X$ is $T_1$ and paracompact (see  \cite[Theorem 3.13]{BF14}).
\item[(16)] Since $\Sigma^{\Pi^{sp}(X,x_0)}\subseteq \Sigma^{tSpan}$.
\end{enumerate} 

In order to investigate further the above diagram, we raise some questions in the following which we are interested in giving answers to them.\\
\begin{enumerate}
\item[(Q1)] Is there a necessary and sufficient condition on $X$ for the equality $\pi_1^{gcov}(X,x_0)=\pi_1^{wh}(X,x_0)$?
\item[(Q2)] Is there a space $X$ for which the strict inequality $\pi_1^{\widetilde{\pi}^{sp}(X,x_0)}(X,x_0) \prec \pi_1^{gcov}(X,x_0)$ holds?
\item[(Q3)] Is there a space $X$ for which the strict inequality $\pi_1^{pSpan}(X,x_0) \prec \pi_1^{\widetilde{\pi}^{sp}(X,x_0)}(X,x_0)$ holds?
\item[(Q4)] Is there a space $X$ for which the strict inequality $\pi_1^{\pi^{sp}(X,x_0)}(X,x_0) \prec \pi_1^{pSpan}(X,x_0)$ holds?
\item[(Q5)] Is there a necessary and sufficient condition on $X$ for the equality $\pi_1^{pSpan}(X,x_0)=\pi_1^{\tau}(X,x_0)$?
\item[(Q6)] Is there a space $X$ for which the strict inequality $\pi_1^{pSpan}(X,x_0) \prec \pi_1^{\tau}(X,x_0)$ holds?
\item[(Q7)] Is there a space $X$ for which the strict inequality $\pi_1^{tSpan}(X,x_0) \prec \pi_1^{Span}(X,x_0)$ holds?
\item[(Q8)] Is there a space $X$ for which the strict inequality $\pi_1^{tSpan}(X,x_0) \prec \pi_1^{{\Pi}^{sp}(X,x_0)}(X,x_0)$ holds?
\item[(Q9)] Is there a space $X$ for which the strict inequality $ \pi_1^{{\Pi}^{sp}(X,x_0)}(X,x_0) \prec \pi_1^{{\pi}^{sp}(X,x_0)}(X,x_0)$ holds?
\item[(Q10)] Is there a space $X$ for which the strict inequality $\pi_1^{sh}(X,x_0) \prec \pi_1^{tSpan}(X,x_0) $ holds?
\end{enumerate}

\section{Topologized Fundamental Groups as Topological Groups}

At the beginning of the 21st century, it was proved in a wrong way that the fundamental group with compact-open quotient topology is a topological group (see \cite{Biss,F9,F11,B11}). Later Fabel \cite{F9,F11} and Brazas \cite{B13} by giving some examples showed that $\pi_1^{qtop}(X,x_0)$ fails to be a topological group, in general. In fact they showed that the group multiplication in $\pi_1^{qtop}(X,x_0)$ is not necessarily continuous. Brazas \cite{B13} mentioned that $\pi_1^{qtop}(X,x_0)$ is a quasitopological group in the sense of \cite{Ar}, that is, a group with a topology such that inversion and all translations are continuous. 
Therefore, it seems interesting to find out when $\pi_1^{qtop}(X,x_0)$ is a topological group.
Calcut and McCarthy \cite{CM} proved that the topology of fundamental group of a locally path connected and semi-locally simply connected space is discrete and so this space has the quasitopological fundamental group as topological group. Also, Brazas \cite{B13} introduced a new topology on fundamental groups made them topological groups and denoted it by $\pi_1^{\tau}(X,x)$.  In fact, the topology of $\pi_1^{\tau}(X,x)$ is obtained by removing the smallest number of open sets from the topology of $\pi_1^{qtop}(X,x_0)$ in order to make it a topological group. Hence the topology of $\pi_1^{\tau}(X,x)$ is coarser than the topology of $\pi_1^{qtop}(X,x_0)$, and if $\pi_1^{qtop}(X,x_0)$ is a topological group, then $\pi_1^{qtop}(X,x_0)\cong\pi_1^{\tau}(X,x)$ as topological groups \cite{B13}. Moreover, Brazas and Fabel \cite{BF15}, Torabi et al. \cite{TPM}, Nasri et al. \cite{N} and Torabi \cite{T21} investigated some conditions under which $\pi_1^{qtop}(X,x_0)$ is a topological group.

For the fundamental group with the whisker topology $\pi_1^{wh}(X,x_0)$, it is known that $\pi_1^{wh}(X,x_0)$ is a homogeneous space and also a left topological group which is not a right topological group, in general. Moreover, the inversion in $\pi_1^{wh}(X,x_0)$ is not continuous, in general. In fact, the inversion in $\pi_1^{wh}(X,x_0)$ is continuous if and only if $\pi_1^{wh}(X,x_0)$ is a topological group. Note that since $\pi_1^{wh}(X,x_0)$ is a subgroup topology with respect to 
$\Sigma^{wh}$, $\pi_1^{wh}(X,x_0)$ is a topological group if $\pi_1(X,x_0)$ is an abelian group.
Furthermore, $\pi_1^{wh}(X,x_0)$ is discrete if and only if $X$ is semilocally simply connected at $x_0$ (see \cite{Br12,A16}.

Finally, note that the fundamental group with the lasso and Spanier topologies which are the same, $\pi_1^{\mathrm{lasso}} (X, x_0)=\pi_1^{\mathrm{Span}}(X, x_0)$ (see \cite[Proposition 3.5]{A20}) is a topological group since $\pi_1^{Span}(X, x_0)$ has the subgroup topology induced by Spanier subgroups $\pi(\mathcal{U}, x_0)$ which are normal subgroups of
$\pi_1(X,x_0)$. Moreover, by \cite[Proposition 3.13]{A20} $\pi_1^{\widetilde{\pi}^{sp}(X,x_0)}(X,x_0)$ is a topological group.
 Also, note that since $\pi^{sg}(X, x_0)$, $\pi^{sp}(X, x_0)$ and $\Pi^{sp}(X, x_0)$ are normal subgroups of $\pi_1(X,x_0)$, by Theorem \ref{2.3} $\pi_1^{\pi^{sg}(X,x_0)}(X,x_0)$, $\pi_1^{\pi^{sp}(X,x_0)}(X,x_0)$ and $\pi_1^{\Pi^{sp}(X,x_0)}(X,x_0)$ are topological groups. Moreover, since $ \check{\pi}_1(X,x_0)$ is a topological group, it follows that ${\pi}_1^{sh}(X,x_0)$ is a topological group (see \cite{BF24}).

Therefore, the question\emph{``What kind of topological structure are topologized fundamental groups?''} is still interesting.

In \cite{TPM} Torabi et al.  proved that if $X$ is semilocally $\pi^{sg}(X,x_0)$-connected, then $\pi_1^{qtop}(X,x_0)$ is a topological group. 
 In the following theorem, we give a weaker condition on $X$ under which $\pi_1^{qtop}(X,x_0)$ is a topological group.
\begin{theorem}\label{4.1}
Let $(X,x_0)$ be a pointed topological group and $H=\pi^{gc}(X,x_0)$ and X is semilocally $H$-connected, then $\pi_1^{qtop}(X,x_0)$ is a topological group.
\end{theorem}
\begin{proof}
Since $X$ is semilocally $\pi^{gc}(X,x_0)$-connected by \cite[Corollary 4.7]{A16} the equality $GCOV(X) = COV(X)$ holds. Therefore the two neighbourhood families of subgroups   $\Sigma^{gcov}$ and $\Sigma^{Span}$ are the same. Hence  $\pi_1^{gcov}(X,x_0) = \pi_1^{Span}(X,x_0)$. Since $\pi_1^{Span}(X,x_0)$ is a topological group the result holds.
\end{proof}

The following example shows that $ {\pi }^{sg}(X,x_0) $ may be a proper subgroup of $ {\pi }^{gc}(X,x_0)$. Hence it seems that the condition semilocally $\pi^{gc}(X,x_0)$-connectedness is weaker than the condition semilocally $\pi^{sg}(X,x_0)$-connectedness in \cite[Theorem 4.1]{TPM}
\begin{example}\label{4.2}
Consider the space $ RX $  described in \cite[Definition 7]{FZ13}. By \cite [Theorem 16]{FZ13} $ RX $ is a metric, path connected, locally path connected and homotopically Hausdorff space and so by \cite{FR,V} $ {\pi }^{sg}(RX,x_0)=1 $, which does not admit a generalized universal covering space, i.e, $ {\pi }^{gc}(RX,x_0)\neq1 $ (see \cite[Example 2.5]{A16}).
\end{example}

It is proved in \cite[Theorem 2.6]{A16} that if $(X,x_0)$ is a locally path connected space, then there exists the following chain of subgroups of ${\pi }_1(X,x_0)$.
\[\{e\}\leq {\pi }^s(X,x_0)\leq {\pi }^{sg}(X,x_0)\le \pi ^{gc}(X,x_0)\le \overline{{\ \pi }^{sg}(X,x_0)}\leq {\widetilde{\pi }}^{sp}(X,x_0)\leq {\pi }^{sp}(X,x_0). \]
Also, it is shown in \cite[Corollary 2.8]{TPM} that if $\ov{\pi^{sg}(X,x_0)}$ is a finite index subgroup of $\pi_1^{qtop}(X,x_0)$ and $\pi_1^{qtop}(X,x_0)$ is connected, then $\pi_1^{qtop}(X,x_0)$ is an indiscrete topological group. In the following theorem we are going to extend this result somehow.

\begin{theorem}\label{4.3}
If $[\pi_1(X,x_0):\pi^{gc}(X,x_0)]$ is countable, then $\pi_1^{gcov}(X,x_0)$ has countably many open sets and so it is second countable.
Moreover, if $X$ is connected, locally path connected and $[\pi_1(X,x_0):\pi^{gc}(X,x_0)]$ is finite, then $\pi_1^{qtop}(X,x_0)$ is a second countable topological group.
\end{theorem}
\begin{proof}
By Theorem \ref{3.4} $\pi_1^{gcov}(X,x_0)=\pi_1^{H}(X,x_0)$, where $H=\pi^{gc}(X,x_0)$. It is easy to see that every open set in $\pi_1^{H}(X,x_0)$ is a union of left cosets of $H$. Therefore, since $[\pi_1(X,x_0):\pi^{gc}(X,x_0)]$ is countable, $\pi_1^{H}(X,x_0)$ has countably many open sets and so $\pi_1^{gcov}(X,x_0)$ is second countable. 

For the second part, since $X$ is connected, locally path connected by \cite[Proposition 3.24]{A20} we have $\pi_1^{qtop}(X,x_0) \preccurlyeq \pi_1^{gcov}(X,x_0)$. Since $[\pi_1(X,x_0):\pi^{gc}(X,x_0)]$ is finite, by a similar argument of the first part 
$\pi_1^{gcov}(X,x_0)$ has finitely many open sets and so is  $\pi_1^{qtop}(X,x_0)$. Therefore $\pi_1^{qtop}(X,x_0)$ is compact and so by Ellis' Theorem (see 
\cite[p. ix]{Ar}) $\pi_1^{qtop}(X,x_0)$ is a second countable topological group. 
\end{proof}

The following result is a consequence of Theorem \ref{2.3}.
\begin{theorem}\label{4.4}
Let $H$ be a subgroup of $\pi_1(X,x_0)$. Then $\pi_1^H(X,x_0)$ is a topological group if and only if $H$ is a normal subgroup of $\pi_1(X,x_0)$.
\end{theorem}

\begin{corollary}\label{4.5}
$(i)$ If $\pi_1(X,x_0)$ is a Dedekind group, then $\pi_1^H(X,x_0)$ is a topological group, for every subgroup $H$ of $\pi_1(X,x_0)$.\\
$(ii)$ $\pi_1^{\Pi^{sp}(X,x_0)}(X,x_0)$, $\pi_1^{\pi^{sp}(X,x_0)}(X,x_0)$ and $\pi_1^{\pi^{sg}(X,x_0)}(X,x_0)$ are topological groups.
\end{corollary}

 \end{document}